\documentclass[12pt,reqno]{amsart}

\usepackage{geometry}
\usepackage{hhline}

\usepackage{mathtools}

\usepackage{mathdots}

\usepackage{color}

\usepackage{pdfsync}

\usepackage{enumitem}

\usepackage{wasysym}

\usepackage{amssymb}

\usepackage{mathrsfs}

\usepackage{bbm}

\DeclareMathAlphabet{\mathpzc}{OT1}{pzc}{m}{it}

\usepackage[all]{xy}

\usepackage{tikz}
\usetikzlibrary{arrows,matrix,decorations.pathmorphing,decorations.pathreplacing,positioning,shapes.geometric,shapes.misc,decorations.markings,decorations.fractals,calc,patterns}

\usepackage{tikz-cd}

\usepackage{graphicx}
\usepackage{float}
\usepackage[bottom]{footmisc}
\usepackage{moreenum}
\usepackage{makecell}
\entrymodifiers={+!!<0pt,\fontdimen22\textfont2>}

\usepackage{scalerel,stackengine}
\stackMath
\newcommand\newcheck[1]{%
\savestack{\tmpbox}{\stretchto{%
  \scaleto{%
    \scalerel*[\widthof{\ensuremath{#1}}]{\kern-.6pt\bigwedge\kern-.6pt}%
    {\rule[-\textheight/2]{1ex}{\textheight}}%WIDTH-LIMITED BIG WEDGE
  }{\textheight}% 
}{0.5ex}}%
\stackon[1pt]{#1}{\scalebox{-1}{\tmpbox}}%
}
\stackMath
\newcommand\newhat[1]{%
\savestack{\tmpbox}{\stretchto{%
  \scaleto{%
    \scalerel*[\widthof{\ensuremath{#1}}]{\kern-.6pt\bigwedge\kern-.6pt}%
    {\rule[-\textheight/2]{1ex}{\textheight}}%WIDTH-LIMITED BIG WEDGE
  }{\textheight}% 
}{0.5ex}}%
\stackon[1pt]{#1}{\scalebox{1}{\tmpbox}}%
}

\def\cB{\mathscr{B}}

\def\cM{\mathscr{M}}

\def\cU{\mathscr{U}}
\def\cV{\mathscr{V}}
\def\cW{\mathscr{W}}
\def\cX{\mathscr{X}}
\def\cY{\mathscr{Y}}

\def\BZ{\mathbb{Z}}

\mathchardef\mhyphen="2D

\def\add{\operatorname{add}}

\def\adots{\mathinner{\mkern1mu\raise1.0pt\vbox{\kern7.0pt\hbox{.}}\mkern2mu\raise5.0pt\hbox{.}\mkern2mu\raise9.0pt\hbox{.}\mkern1mu}}

\def\dddots{\mathinner{\mkern1mu\raise10.0pt\vbox{\kern7.0pt\hbox{.}}\mkern2mu\raise5.3pt\hbox{.}\mkern2mu\raise1.0pt\hbox{.}\mkern1mu}}
\def\dddotssmall{\mathinner{\mkern1mu\raise7.0pt\vbox{\kern7.0pt\hbox{.}}\mkern-1mu\raise4pt\hbox{.}\mkern-1mu\raise1.0pt\hbox{.}\mkern1mu}}

\def\dtors{d\mbox{-tors}}
\def\dual{\operatorname{D}}

\def\H{\operatorname{H}}

\def\Hom{\operatorname{Hom}}

\def\K{\operatorname{K}}

\def\mod{\operatorname{mod}}

\def\PSL2{\operatorname{PSL}_2}

\def\SL2{\operatorname{SL}_2}

\numberwithin{equation}{section}

\newtheorem{Lemma}{Lemma}[section]
\newtheorem{Theorem}[Lemma]{Theorem}

\theoremstyle{definition}

\newtheorem{Setup}[Lemma]{Setup}

\newtheorem{Remark}[Lemma]{Remark}

\newtheorem{Example}[Lemma]{Example}

\theoremstyle{theorem}
\newtheorem{ThmIntro}{Theorem}

\theoremstyle{definition}

\newtheorem*{bfhpg*}{}

\usepackage{amsthm}
\usepackage{thmtools}
\declaretheoremstyle[
notefont=\bfseries, notebraces={}{},
bodyfont=\normalfont,
headformat=\NUMBER~\NOTE,
headpunct={}
]{foobar}

  {\begin{list}{}{%
    \settowidth{\labelwidth}{\textbf{#1:}}%
    \setlength{\leftmargin}{\labelwidth}\addtolength{\leftmargin}{\labelsep}}}%
  {\end{list}}

\makeatletter
\@namedef{subjclassname@2020}{%
  \textup{2020} Mathematics Subject Classification}
\makeatother

\begin{document}

\setlength{\parindent}{0pt}
\setlength{\parskip}{7pt}

\title[Splitting higher torsion classes]{Splitting torsion classes in higher homological algebra}

\author{Erlend D.\ B\o rve}

\address{Department of Mathematics, Aarhus University, Ny Munkegade 118, 8000 Aarhus C, Denmark}
\email{erlend.d.borve@math.au.dk}

%\urladdr{https://bervinator.github.io}

\author{Peter J\o rgensen}

\address{Department of Mathematics, Aarhus University, Ny Munkegade 118, 8000 Aarhus C, Denmark}
\email{peter.jorgensen@math.au.dk}

\urladdr{https://sites.google.com/view/peterjorgensen}

\author{Mads Hustad Sand\o y}

\address{Department of Mathematics, Aarhus University, Ny Munkegade 118, 8000 Aarhus C, Denmark}
\email{mads.sandoy@ntnu.no}

\urladdr{https://sites.google.com/view/mads-h-sandoy-math}

%\thanks{Date: \today. A thank you would go here}



\begin{abstract} 

Let $d \geqslant 1$ be an integer, $\cM$ a suitable $d$-abelian category.  There is a notion of higher torsion classes in $\cM\!$, also known as $d$-torsion classes.  

\medskip
\noindent
We generalise a classic theorem of Hoshino by providing criteria for a $d$-torsion class $\cU$ to be splitting,  notably that $\cU$ is stable under $\tau_d^-$, the inverse $d$-Auslander--Reiten translation.

\medskip
\noindent
We apply the criteria to show that under additional assumptions, which are satisfied in Dynkin type $A$, the elements of the spine of the lattice of $d$-torsion classes are precisely the splitting $d$-torsion classes.  The spine consists of the elements which belong to a chain of maximal length. 

\end{abstract}

\maketitle

\setcounter{section}{-1}
\section{Introduction}
\label{sec:introduction}

Higher homological algebra was founded by Iyama, see \cite{Iyama_Auslander-correspondence}, \cite{Iyama_Cluster-tilting-for-higher-Auslander-algebras}, \cite{Iyama_Higher-dimensional-AR-theory}.  A subsequent key step was taken by Jasso who introduced $d$-abelian categories for integers $d \geqslant 1$, see  \cite[def.\ 3.1]{Jasso_n-abelian}.  A $1$-abelian category is just an abelian category in the classic sense, while a $d$-abelian category has complexes of length $d$ instead of kernels and cokernels.

Higher torsion classes in $d$-abelian categories generalise torsion classes in classic abelian categories.  They are also known as $d$-torsion classes and were introduced in \cite[def.\ 1.1]{Jorgensen_Torsion-classes-in-higher-homological-algebra}, where splitting $d$-torsion classes were defined too.  See Remark \ref{rmk:d-torsion_classes}.  

Our first main result is the following generalisation of the classic characterisation of splitting torsion classes due to Hoshino; see \cite[prop.\ 1]{Hoshino_On-splitting-torsion-theories-induced-by-tilting-modules} and \cite[prop.\ VI.1.7]{Assem-Simson-Skowronski_Elements-I}.  Let $\cM$ be a $d$-abelian category arising as a $d$-cluster tilting subcategory of $\mod\,A$ for a finite dimensional algebra $A$ over an algebraically closed field $k$; see \cite[def.\ 2.2]{Iyama_Higher-dimensional-AR-theory} and \cite[def.\ 3.14]{Jasso_n-abelian}.

\begin{ThmIntro}
[= Theorem \ref{thm:generalisation_of_ASS_VI1.7}]
\label{thm:A}
Let $\cU \subseteq \cM$ be a $d$-torsion class.  The following two conditions are equivalent:
\begin{enumerate}
\setlength\itemsep{4pt}

  \item  $\cU$ is splitting.
  
  \item  $\tau_d^-\cU \subseteq \cU\!$, where $\tau_d^-$ is the inverse $d$-Auslander--Reiten translation. 

\end{enumerate}
Now assume $d \geqslant 2$.  Then conditions (i) and (ii) are also equivalent to the following:
\begin{enumerate}
\setcounter{enumi}{2}
\setlength\itemsep{4pt}

  \item  If $\delta = 0 \xrightarrow{} U \xrightarrow{} M \xrightarrow{} V^1 \xrightarrow{} \cdots \xrightarrow{} V^d \xrightarrow{} 0$ is a canonical $d$-exact sequence with respect to $\cU\!$, then
\[  
  \overline{ \Hom }_A( \widetilde{ U },\tau_dV^{ d-1 } ) \xrightarrow{} \overline{ \Hom }_A( \widetilde{ U },\tau_dV^{ d } )
\]  
is surjective for each $\widetilde{ U } \in \cU\!$, where $\tau_d$ is the $d$-Auslander--Reiten translation and $\overline{ \Hom }_A$ is injectively stabilised $\Hom_A$.  

\end{enumerate}
\end{ThmIntro}

Note that the classic characterisation of splitting torsion classes also provides a condition on the corresponding torsion free classes.  This is not possible for $d$-torsion classes, which do not in general have corresponding $d$-torsion free classes.  However, a substitute is provided by condition (iii) in Theorem \ref{thm:A}.

Our second main result concerns the partially ordered set of $d$-torsion classes.  It was proved by August, Haugland, Jacobsen, Kvamme, Palu, and Treffinger that it is a lattice; see \cite[thm.\ 4.3]{August-et-al_A-characterisation-of-higher-torsion-classes}.  The spine of a lattice was originally defined for trim lattices by Thomas in \cite[p.\ 257]{Thomas_An-analogue-of-distributivity-for-ungraded-lattices}; it consists of the elements which belong to a chain of maximal length.  The spine of the lattice of classic torsion classes has recently been investigated by Asai, Iyama, Mousavand, and Paquette; see \cite[thm.\ 1.7]{Asai-Iyama-Mousavand-Paquette_Brick-splitting-torsion-pairs-and-trim-lattices}.  We show the following result in the higher case.

\begin{ThmIntro}
[= Theorem \ref{thm:spine}]
\label{thm:B}
Assume that $\cM$ has no cycles in the sense of Remark \ref{rmk:cycles} and that it has only finitely many isomorphism classes of indecomposable modules.  

Then the elements of the spine of the lattice of $d$-torsion classes of $\cM$ are precisely the splitting $d$-torsion classes.
\end{ThmIntro}

Theorem \ref{thm:B} applies in particular to Dynkin type $A$; see Example \ref{exa:Type_A}.

We refer to \cite{Asadollahi-Jorgensen-Schroll-Treffinger_On-higher-torsion-classes}, \cite{August-et-al_A-characterisation-of-higher-torsion-classes}, \cite{August-et-al_Higher-torsion-classes-tau-d-tilting-theory-and-silting-complexes}, \cite{Fedele-Jorgensen-Shah_The-index-in-d-exact-categories}, \cite{Kvamme_Higher-extension-closure-and-d-exact-categories}, \cite{Segovia_P-comma-phi-Tamari-and-higher-torsion-lattices-of-type-A}, \cite{Segovia_P-comma-phi-Tamari-lattices} for other recent contributions to the theory of $d$-torsion classes.

We end the introduction with a blanket setup and a few background items.

\begin{Setup}
The following will be fixed throughout.
\begin{enumerate}
\setlength\itemsep{4pt}

  \item  $k$ is an algebraically closed field and $d \geqslant 1$ is an integer.
  
  \item  $A$ is a finite dimensional $k$-algebra and $\mod A$ is the abelian category of finite dimensional right $A$-modules.
  
  \item  $\dual( - ) = \Hom_k( -,k )$ is the $k$-linear duality functor.

  \item  $\cM \subseteq \mod\,A$ is a $d$-cluster tilting subcategory; see \cite[def.\ 2.2]{Iyama_Higher-dimensional-AR-theory} and \cite[def.\ 3.14]{Jasso_n-abelian}.   
  
  \item  If $\cX \subseteq \cM$ is an additive subcategory, then we write
\[
  \cX^{ \perp }
  = \{\, Y \in \cM \,|\, \Hom_A( \cX,Y ) = 0 \,\}.
\]

\end{enumerate}
\end{Setup}

\begin{Remark}
[$d$-abelian categories]
Under the setup, $\cM$ is a $d$-abelian category by \cite[thm.\ 3.16]{Jasso_n-abelian}.  The notion of $d$-exact sequences in $\cM$ will be used repeatedly; see \cite[def.\ 2.4]{Jasso_n-abelian}.  It follows from \cite[props.\ 2.2 and 2.6]{Jasso-Kvamme_Introduction-to_higher_AR-theory} that a $d$-exact sequence in $\cM$ is simply an exact sequence in $\mod\,A$,
\[
  0
  \xrightarrow{} M^0
  \xrightarrow{} M^1
  \xrightarrow{} \cdots
  \xrightarrow{} M^d
  \xrightarrow{} M^{ d+1 }
  \xrightarrow{} 0,
\]
with $M^i \in \cM$ for each $i$.
\end{Remark}

\begin{Remark}
[cycles]
\label{rmk:cycles}
Recall from \cite[p.\ 313]{Auslander-Reiten-Smalo_Representation-theory-of-Artin-algebras} that a cycle from $M$ to $M$ in $\cM$ is a sequence $M \xrightarrow{ \mu_0 } M_1 \xrightarrow{ \mu_1 } \cdots \xrightarrow{ \mu_{ t-2 } } M_{ t-1 } \xrightarrow{ \mu_{ t-1 } } M$ with $t \geqslant 1$, where $M$ and each $M_i$ is indecomposable in $\cM\!$, and each $\mu_i$ is not zero and not an isomorphism.  Note that if $t = 1$, then the cycle has the form $M \xrightarrow{ \mu_0 } M$.
\end{Remark}

\begin{Remark}
[$d$-torsion classes]
\label{rmk:d-torsion_classes}
${\;}$
\begin{enumerate}
\setlength\itemsep{4pt}

  \item  Let $\cU \subseteq \cM$ be an additive subcategory.  Let $M \in \cM$ be given.  A canonical $d$-exact sequence of $M$ with respect to $\cU$ is a $d$-exact sequence 
\begin{equation}
\label{equ:canonical_sequence}
  0
  \xrightarrow{} U
  \xrightarrow{ \upsilon } M
  \xrightarrow{} V^1
  \xrightarrow{} \cdots
  \xrightarrow{} V^d
  \xrightarrow{} 0
\end{equation}
in $\cM$ with $U \in \cU$ such that
\[
  0
  \xrightarrow{} \Hom_A( \widetilde{ U },V^1 )
  \xrightarrow{} \cdots 
  \xrightarrow{} \Hom_A( \widetilde{ U },V^d )
  \xrightarrow{} 0
\]
is exact for each $\widetilde{ U } \in \cU\!$.  In such a sequence, $M$ determines $\upsilon$ up to isomorphism since $\upsilon$ is a $\cU$-cover of $M$; see \cite[lem.\ 2.7(i)]{Jorgensen_Torsion-classes-in-higher-homological-algebra}.

  \item  A $d$-torsion class $\cU \subseteq \cM$ is an additive subcategory such that each $M \in \cM$ has a canonical $d$-exact sequence with respect to $\cU\!$.
  
  \item  A $d$-torsion class $\cU \subseteq \cM$ is called splitting if each $M \in \cM$ has a canonical $d$-exact sequence \eqref{equ:canonical_sequence} with respect to $\cU$ where $\upsilon$ is a split monomorphism.  In this case, every canonical $d$-exact sequence with respect to $\cU$ has $\upsilon$ a split monomorphism because $M$ determines $\upsilon$ up to isomorphism.

  \item  The set of $d$-torsion classes in $\cM$ is denoted $\dtors\,\cM$.  It is clearly a partially ordered set under inclusion.  This partially ordered set is, in fact, a lattice; that is, it has meets and joins.  This holds by \cite[thm.\ 4.3]{August-et-al_A-characterisation-of-higher-torsion-classes}, which also states that meets are given by intersection.

\end{enumerate}
\end{Remark}

\begin{Remark}
[$d$-Auslander--Reiten theory]
\label{rmk:tau_d}
${\;}$
\begin{enumerate}
\setlength\itemsep{4pt}

  \item  The projective stable module category $\underline{ \mod }\,A$ is the na\"{i}ve quotient of $\mod A$ by the ideal of homomorphisms factoring through a projective module, and its $\Hom$ functor is $\underline{ \Hom }_A$.  There is a full subcategory $\underline{ \cM } \subseteq \underline{ \mod }\,A$ induced by $\cM \subseteq \mod\,A$.
  
  \item  The injective stable module category $\overline{ \mod }\,A$ is the na\"{i}ve quotient of $\mod A$ by the ideal of homomorphisms factoring through an injective module, and its $\Hom$ functor is $\overline{ \Hom }_A$.  There is a full subcategory $\overline{ \cM } \subseteq \overline{ \mod }\,A$ induced by $\cM \subseteq \mod\,A$.

  \item  The $d$-Auslander--Reiten translation $\tau_d$ and its left adjoint $\tau_d^-$ are functors
\[
  \begin{tikzcd}
	{\underline{ \mod }\,A} & {\overline{ \mod }\,A}\lefteqn{,}
	\arrow["{\tau_d}"', shift right=1, from=1-1, to=1-2]
	\arrow["{\tau_d^-}"', shift right=1, from=1-2, to=1-1]
  \end{tikzcd}
\]  
see \cite[thm.\ 2.3.1(1)]{Iyama_Auslander-correspondence} and \cite[thm.\ 1.4.1]{Iyama_Higher-dimensional-AR-theory}.  They restrict to pseudo-inverse functors
\[
  \begin{tikzcd}
	{\underline{ \cM }} & {\overline{ \cM }\!.}
	\arrow["{\tau_d}"', shift right=1, from=1-1, to=1-2]
	\arrow["{\tau_d^-}"', shift right=1, from=1-2, to=1-1]
  \end{tikzcd}
\]  

  \item  A $d$-exact sequence $\delta$ in $\cM$ gives rise to co- and contravariant defect functors $\delta_*$ and $\delta^*$ defined in \cite[def.\ 3.1]{Jasso-Kvamme_Introduction-to_higher_AR-theory}.  By \cite[thm.\ 3.8]{Jasso-Kvamme_Introduction-to_higher_AR-theory} they satisfy the higher defect formula
\begin{equation}
\label{equ:defect_formula}
  \dual\!\delta^*( - ) \cong \delta_*\big( \tau_d( - ) \big),
\end{equation}
where both sides are functors defined on $\underline{ \cM }$, and its dual
\begin{equation}
\label{equ:dual_defect_formula}
  \dual\!\delta_*( - ) \cong \delta^*\big( \tau^-_d( - ) \big),
\end{equation}
where both sides are functors defined on $\overline{ \cM }\!$.

  \item  If $M \in \cM$ is given, then $\tau_d M$ is only defined up to isomorphism in $\overline{ \cM }$, but we make it defined up to isomorphism in $\cM$ by picking it not to have any non-zero injective summands.  Similarly, $\tau_d^- M$ is only defined up to isomorphism in $\underline{ \cM }$, but we make it defined up to isomorphism in $\cM$ by picking it not to have any non-zero projective summands.  

\medskip
\noindent
With this stipulation, $\tau_d$ and $\tau_d^-$ induce bijections between the isomorphism classes of non-projective indecomposable modules in $\cM$ and the isomorphism classes of non-injective indecomposable modules in $\cM\!$, see \cite[thm.\ 2.3]{Iyama_Higher-dimensional-AR-theory}.

\end{enumerate}
\end{Remark}

\section{Homological characterisation of splitting $d$-torsion classes}
\label{sec:splitting}

\begin{Lemma}
\label{lem:criterion_for_splitting}
Let $\cU \subseteq \cM$ be an additive subcategory.  The following conditions are equivalent:
\begin{enumerate}
\setlength\itemsep{4pt}

  \item  $\cU$ is a splitting $d$-torsion class.

  \item  Each $M \in \cM$ can be written $M = U \coprod V$ with $U \in \cU$ and $V \in \cU^{ \perp }$.

  \item  Each indecomposable $M \in \cM$ is either in $\cU$ or in $\cU^{ \perp }$.
  
\end{enumerate}
\end{Lemma}

\begin{proof}
(i)$\Rightarrow$(ii):  Assume that (i) is satisfied.  Let $M \in \cM$ be given.  Pick a canonical $d$-exact sequence \eqref{equ:canonical_sequence} of $M$ with respect to $\cU$ where $\upsilon$ is a split monomorphism.  Up to isomorphism, we can write $M = U \coprod V$ and view $\upsilon$ as the coproduct inclusion $U \xrightarrow{} M$.  Since $\upsilon$ is a $\cU$-cover by \cite[lem.\ 2.7(i)]{Jorgensen_Torsion-classes-in-higher-homological-algebra}, this implies $V \in \cU^{ \perp }$ which proves (ii).

(ii)$\Rightarrow$(i):  Assume that (ii) is satisfied.  Let $M \in \cM$ be given.  Write $M = U \coprod V$ with $U \in \cU$ and $V \in \cU^{ \perp }$.  Let $U \xrightarrow{ \iota } M$ be the coproduct inclusion.  There is a $d$-exact sequence 
\begin{equation}
\label{equ:lem:criterion_for_splitting:20}
  0
  \xrightarrow{} U
  \xrightarrow{ \iota } M
  \xrightarrow{} V
  \xrightarrow{} 0
  \xrightarrow{} \cdots
  \xrightarrow{} 0
\end{equation}
in $\cM$ by \cite[prop.\ 2.2]{Jasso-Kvamme_Introduction-to_higher_AR-theory} which can be used as the canonical $d$-exact sequence \eqref{equ:canonical_sequence}.  Since $\iota$ is a split monomorphism, this proves (i).

(ii)$\Leftrightarrow$(iii) is immediate because $\cM$ is a Krull--Schmidt category.
\end{proof}

\begin{Lemma}
\label{lem:criterion_for_contractible}
Let $\cB$ be an additive category and let $\K( \cB )$ be the homotopy category of chain complexes over $\cB\!$.  Let $B$ be a bounded chain complex over $\cB\!$.  The following conditions are equivalent:
\begin{enumerate}
\setlength\itemsep{4pt}

  \item  $B$ is isomorphic to zero in $\K( \cB )$.

  \item  For each $\widetilde{ B } \in \cB$ the complex $\Hom_{ \cB }( \widetilde{ B },B )$ is exact.
  
  \item  For each $\widetilde{ B } \in \cB$ the complex $\Hom_{ \cB }( B,\widetilde{ B } )$ is exact.

\end{enumerate}
\end{Lemma}

\begin{proof}
If $i \in \BZ$ then \cite[thm.\ 2.3.10]{Christensen-Foxby-Holm_Derived-category-methods} implies $\H_i\!\Hom_{ \cB }( \widetilde{ B },B ) \cong \Hom_{ \K( \cB ) }( \Sigma^i\widetilde{ B },B )$, where $\Sigma$ is the suspension functor of $\K( \cB )$ and $\widetilde{ B }$ is viewed as a complex concentrated in degree zero.  Hence (ii) is equivalent to the following condition.
\begin{itemize}
\setlength\itemsep{4pt}

  \item[(ii')]  For each $i \in \BZ$ and $\widetilde{ B } \in \cB$ we have $\Hom_{ \K( \cB ) }( \Sigma^i\widetilde{ B },B ) = 0$.
  
\end{itemize}
To prove (i)$\Leftrightarrow$(ii) it is hence enough to prove (i)$\Leftrightarrow$(ii'), and the implication (i)$\Rightarrow$(ii') is clear.  To prove (ii')$\Rightarrow$(i), note that in the triangulated category $\K( \cB )$, the bounded complex $B$ can be built from objects of the form $\Sigma^i\widetilde{ B }$ using finitely many hard truncation triangles.  Hence (ii') implies $\Hom_{ \K( \cB ) }( B,B ) = 0$, which is equivalent to (i).

The equivalence (i)$\Leftrightarrow$(iii) is proved similarly.
\end{proof}

\begin{Lemma}
\label{lem:splitting_implies_Ext^d_zero}
Let $\cU \subseteq \cM$ be a splitting $d$-torsion class.  Let $U \in \cU$ and $V \in \cU^{ \perp }$ be given.  Then each $d$-exact sequence in $\cM$ of the form
\begin{equation}
\label{equ:lem:splitting_implies_Ext^d_zero:10}
  0
  \xrightarrow{} U
  \xrightarrow{} M^1
  \xrightarrow{} \cdots
  \xrightarrow{} M^d
  \xrightarrow{} V
  \xrightarrow{} 0
\end{equation}
is split exact.
\end{Lemma}

\begin{proof}
By Lemma \ref{lem:criterion_for_splitting}(ii) we can write $M^i = U^i \oplus V^i$ with $U^i \in \cU$ and $V^i \in \cU^{ \perp }$ for each $i$.  Hence \eqref{equ:lem:splitting_implies_Ext^d_zero:10} can be spun into the following commutative diagram where the vertical morphisms are inclusions and projections.
\[
  \begin{tikzcd}
	0 & U & {U^1} & \cdots & {U^d} & 0 & 0 \\
	0 & U & {U^1 \coprod V^1} & \cdots & {U^d \coprod V^d} & V & 0 \\
	0 & 0 & {V^1} & \cdots & {V^d} & V & 0
	\arrow[from=1-1, to=1-2]
	\arrow[from=1-2, to=1-3]
	\arrow[from=1-2, to=2-2]
	\arrow[from=1-3, to=1-4]
	\arrow[from=1-3, to=2-3]
	\arrow[from=1-4, to=1-5]
	\arrow[from=1-5, to=1-6]
	\arrow[from=1-5, to=2-5]
	\arrow[from=1-6, to=1-7]
	\arrow[from=1-6, to=2-6]
	\arrow[from=2-1, to=2-2]
	\arrow[from=2-2, to=2-3]
	\arrow[from=2-2, to=3-2]
	\arrow[from=2-3, to=2-4]
	\arrow[from=2-3, to=3-3]
	\arrow[from=2-4, to=2-5]
	\arrow[from=2-5, to=2-6]
	\arrow[from=2-5, to=3-5]
	\arrow[from=2-6, to=2-7]
	\arrow[from=2-6, to=3-6]
	\arrow[from=3-1, to=3-2]
	\arrow[from=3-2, to=3-3]
	\arrow[from=3-3, to=3-4]
	\arrow[from=3-4, to=3-5]
	\arrow[from=3-5, to=3-6]
	\arrow[from=3-6, to=3-7]
  \end{tikzcd}
\]
Extending by zeroes to the left and right, each line in the diagram is a chain complex over $\cM\!$.  The whole diagram (rotated anticlockwise by $90$ degrees) is a semisplit short exact sequence
\begin{equation}
\label{equ:lem:splitting_implies_Ext^d_zero:20}
  0 \xrightarrow{} U^* \xrightarrow{} E^* \xrightarrow{} V^* \xrightarrow{} 0
\end{equation}
of chain complexes over $\cM\!$.

Let $\widetilde{ U } \in \cU$ be given.  Then the complex
\begin{equation}
\label{equ:lem:splitting_implies_Ext^d_zero:30}
  \Hom_A( \widetilde{ U },V^* )
\end{equation}
is exact; in fact, it is zero since each object in $V^*$ is in $\cU^{ \perp }$.  The complex
\begin{equation}
\label{equ:lem:splitting_implies_Ext^d_zero:40}
  \Hom_A( \widetilde{ U },E^* )
\end{equation}
is also exact because the sequence
\[
  0
  \xrightarrow{} \Hom_A( \widetilde{ U },U )
  \xrightarrow{} \Hom_A( \widetilde{ U },U^1 \coprod V^1 )
  \xrightarrow{} \cdots
  \xrightarrow{} \Hom_A( \widetilde{ U },U^d \coprod V^d )
  \xrightarrow{} \Hom_A( \widetilde{ U },V )  
\]
is exact since $E^*$ is a $d$-exact sequence, and because we have $\Hom_A( \widetilde{ U },V ) = 0$ by $\widetilde{ U } \in \cU$ and $V \in \cU^{ \perp }$.  Since the complexes \eqref{equ:lem:splitting_implies_Ext^d_zero:30} and \eqref{equ:lem:splitting_implies_Ext^d_zero:40} are exact, the long exact sequence induced by \eqref{equ:lem:splitting_implies_Ext^d_zero:20} implies that $\Hom_A( \widetilde{ U },U^* )$ is exact.  Since this holds for each $\widetilde{ U } \in \cU$ while each object in $U^*$ is in $\cU$, Lemma \ref{lem:criterion_for_contractible} implies that $U^*$ is isomorphic to zero in $\K( \cU )$, hence isomorphic to zero in $\K( \mod A )$.

A symmetric argument using $\Hom_A( -,\widetilde{ V } )$ with $\widetilde{ V } \in \cU^{ \perp }$ shows $V^*$ isomorphic to zero in $K( \mod A )$.

Since \eqref{equ:lem:splitting_implies_Ext^d_zero:20} is a semisplit short exact sequence of chain complexes over $\cM$ and hence over $\mod A$, it induces a triangle $U^* \xrightarrow{} E^* \xrightarrow{} V^* \xrightarrow{} \Sigma U^*$ in the triangulated category $\K( \mod A )$.  Hence $E^*$ is isomorphic to zero in $\K( \mod A )$ because so are $U^*$ and $V^*$.  But $E^*$ is \eqref{equ:lem:splitting_implies_Ext^d_zero:10} extended by zeroes to the left and right, so \eqref{equ:lem:splitting_implies_Ext^d_zero:10} is split exact.
\end{proof}

\begin{Lemma}
\label{lem:splitting_vs_defects}
Let $\cU \subseteq \cM$ be a $d$-torsion class.  The following conditions are equivalent:
\begin{enumerate}
\setlength\itemsep{4pt}

  \item  $\cU$ is splitting.
  
  \item  If $\delta$ is a canonical $d$-exact sequence with respect to $\cU$ then $\delta_*( \cU ) = 0$.

  \item  If $\delta$ is a canonical $d$-exact sequence with respect to $\cU$ then $\delta^*( \tau_d^-\cU ) = 0$.

\end{enumerate}
\end{Lemma}

\begin{proof}
We will use the notation
\[
  \delta 
  \;\;=\;\; 0
  \xrightarrow{} U
  \xrightarrow{ \upsilon } M
  \xrightarrow{} V^1
  \xrightarrow{} \cdots
  \xrightarrow{} V^d
  \xrightarrow{} 0.
\]
By \cite[def.\ 3.1]{Jasso-Kvamme_Introduction-to_higher_AR-theory} the covariant defect functor $\delta_*$ is defined by the exact sequence
\[
  \Hom_A( M,- )
  \xrightarrow{ \upsilon^* }
  \Hom_A( U,- )
  \xrightarrow{}
  \delta_*( - )
  \xrightarrow{}
  0.
\]

(i)$\Rightarrow$(ii):  Assume that (i) is satisfied.  For each canonical $d$-exact sequence $\delta$, the homomorphism $\upsilon$ is a split monomorphism whence $\upsilon^*$ is surjective, so we have $\delta_* = 0$ and in particular $\delta_*( \cU ) = 0$.  This proves (ii).

(ii)$\Rightarrow$(i):  Assume that (ii) is satisfied.   For each canonical $d$-exact sequence $\delta$, we have $\delta_*( U ) = 0$ in particular, whence
\[
  \Hom_A( M,U )
  \xrightarrow{ \Hom_A( \upsilon,U ) }
  \Hom_A( U,U )
\]  
is surjective so $\upsilon$ is a split monomorphism.  This proves (i).

(ii)$\Leftrightarrow$(iii) holds by the dual defect formula, Equation \eqref{equ:dual_defect_formula}.
\end{proof}

\begin{Lemma}
\label{lem:surjective_simultaneously}
Let $Y \stackrel{ \upsilon }{ \twoheadrightarrow } Y''$ be a surjection in $\mod\,A$.  For each $X \in \mod\,A$, the induced homomorphisms
\[
  \Hom_A( X,Y ) \xrightarrow{ \alpha } \Hom_A( X,Y'' )
\]
and
\[
  \underline{ \Hom }_A( X,Y ) \xrightarrow{ \beta } \underline{ \Hom }_A( X,Y'' )
\]
are surjective simultaneously.
\end{Lemma}

\begin{proof}
If $\xi$ is a homomorphism in $\Hom_A$, then $\underline{ \xi }$ will denote its class in $\underline{ \Hom }_A$.  

Suppose $\alpha$ is surjective.  Let $\underline{ \xi }'' \in \underline{ \Hom }_A( X,Y'' )$ be given.  Pick $\xi \in \Hom_A( X,Y )$ such that $\alpha( \xi ) = \xi''$.  Then $\beta( \underline{ \xi } ) = \underline{ \xi }''$.  This proves $\beta$ is surjective.

Suppose $\beta$ is surjective.  Let $\xi'' \in \Hom_A( X,Y'' )$ be given.  Pick $\underline{ \xi } \in \underline{ \Hom }_A( X,Y )$ such that $\beta( \underline{ \xi } ) = \underline{ \xi }''$.  This equation means that $\alpha( \xi )$ and $\xi''$ become equal when projected to $\underline{ \Hom }_A( X,Y'' )$, whence $\alpha( \xi ) - \xi''$ factorises through a projective module $P$ as follows.
\[
\begin{tikzcd}
	&& Y \\
	X & P \\
	&& {Y''}
	\arrow["\upsilon", two heads, from=1-3, to=3-3]
	\arrow["\varphi", dashed, from=2-1, to=1-3]
	\arrow[from=2-1, to=2-2]
	\arrow["{\alpha( \xi ) - \xi''}"', from=2-1, to=3-3]
	\arrow[dotted, from=2-2, to=1-3]
	\arrow[from=2-2, to=3-3]
\end{tikzcd}
\]
The dotted arrow exists because $P$ is projective and $\upsilon$ is surjective.  It induces the dashed arrow $\varphi$ whence $\alpha( \xi ) - \xi'' = \upsilon\varphi$.  Hence we have $\xi'' = \alpha( \xi ) - \upsilon\varphi = \alpha( \xi - \varphi )$.  This proves $\alpha$ is surjective.
\end{proof}

\begin{Theorem}
\label{thm:generalisation_of_ASS_VI1.7}
Let $\cU \subseteq \cM$ be a $d$-torsion class.  The following two conditions are equivalent:
\begin{enumerate}
\setlength\itemsep{4pt}

  \item  $\cU$ is splitting.
  
  \item  $\tau_d^-\cU \subseteq \cU\!$.

\end{enumerate}
Now assume $d \geqslant 2$.  Then conditions (i) and (ii) are also equivalent to the following:
\begin{enumerate}
\setcounter{enumi}{2}
\setlength\itemsep{4pt}

  \item  If $\delta = 0 \xrightarrow{} U \xrightarrow{} M \xrightarrow{} V^1 \xrightarrow{} \cdots \xrightarrow{} V^d \xrightarrow{} 0$ is a canonical $d$-exact sequence with respect to $\cU$, then
\[  
  \overline{ \Hom }_A( \widetilde{ U },\tau_dV^{ d-1 } ) \xrightarrow{} \overline{ \Hom }_A( \widetilde{ U },\tau_dV^{ d } )
\]  
is surjective for each $\widetilde{ U } \in \cU$.  

\end{enumerate}
\end{Theorem}

\begin{proof}
If $d=1$ then the result (i)$\Leftrightarrow$(ii) is contained in \cite[prop.\ VI.1.7]{Assem-Simson-Skowronski_Elements-I}, so we assume $d \geqslant 2$ for the rest of the proof.

(i)$\Rightarrow$(ii):  Assume that (i) is satisfied.  Let $U \in \cU$ be indecomposable.  It is enough to show $\tau_d^-U \in \cU$.  If $U$ is injective then $\tau_d^-U = 0$ is clearly in $\cU\!$, so let us assume that $U$ is non-injective.  By \cite[thm.\ 3.3.1]{Iyama_Higher-dimensional-AR-theory} there is a $d$-Auslander--Reiten sequence in $\cM\!$,
\begin{equation}
\label{equ:thm:generalisation_of_ASS_VI1.7:10}
  0
  \xrightarrow{} U 
  \xrightarrow{} M^1 
  \xrightarrow{} \cdots
  \xrightarrow{} M^d
  \xrightarrow{} \tau_d^-U 
  \xrightarrow{} 0.
\end{equation}
Lemma \ref{lem:criterion_for_splitting}(iii) gives $\tau_d^-U \in \cU$ or $\tau_d^-U \in \cU^{ \perp }$.  But the latter would imply that \eqref{equ:thm:generalisation_of_ASS_VI1.7:10} were split exact by Lemma \ref{lem:splitting_implies_Ext^d_zero}.  This is false by definition of $d$-Auslander--Reiten sequences, see \cite[p.\ 33]{Iyama_Higher-dimensional-AR-theory}, so we conclude $\tau_d^-U \in \cU$.  We have proved (ii).

Before continuing, we claim that condition (i) in the theorem is equivalent to the following:
\begin{itemize}
\setlength\itemsep{4pt}

  \item[(i')]  If $\delta = 0 \xrightarrow{} U \xrightarrow{} M \xrightarrow{} V^1 \xrightarrow{} \cdots \xrightarrow{} V^d \xrightarrow{} 0$ is a canonical $d$-exact sequence with respect to $\cU$, then
\[  
  \Hom_A( \tau_d^-\widetilde{ U },V^{ d-1 } )
  \xrightarrow{}
  \Hom_A( \tau_d^-\widetilde{ U },V^{ d } )
\]  
is surjective for each $\widetilde{ U } \in \cU$.  

\end{itemize}
To show this, observe that when $\delta$ is as in (i'), the contravariant defect functor $\delta^*$ is defined by the exact sequence
\[
  \Hom_A( -,V^{ d-1 } )
  \xrightarrow{}
  \Hom_A( -,V^d )
  \xrightarrow{}
  \delta^*( - )
  \xrightarrow{}
  0
\]
by \cite[def.\ 3.1]{Jasso-Kvamme_Introduction-to_higher_AR-theory}.  Hence the surjectivity required in (i') is equivalent to $\delta^*( \tau_d^-\cU ) = 0$, so (i') is equivalent to Lemma \ref{lem:splitting_vs_defects}, condition (iii).  This is again equivalent to Lemma \ref{lem:splitting_vs_defects}, condition (i), which is equal to condition (i) in the theorem.

(ii)$\Rightarrow$(i'):  This is clear.

(i')$\Leftrightarrow$(iii):  When $\delta$ is as in (i') and $\widetilde{ U } \in \cU$ is given, there is the following commutative diagram where $\alpha$, $\beta$, $\gamma$ are induced by $V^{ d-1 } \xrightarrow{} V^d$, and the vertical isomorphisms are given by Remark \ref{rmk:tau_d}(iii).
\[
\begin{tikzcd}
	{\Hom_A( \tau_d^-\widetilde{ U },V^{ d-1 } )} && {\Hom_A( \tau_d^-\widetilde{ U },V^{ d } )} \\
	{\underline{ \Hom }_A( \tau_d^-\widetilde{ U },V^{ d-1 } )} && {\underline{ \Hom }_A( \tau_d^-\widetilde{ U },V^{ d } )} \\
	{\overline{ \Hom }_A( \widetilde{ U },\tau_dV^{ d-1 } )} && {\overline{ \Hom }_A( \widetilde{ U },\tau_dV^{ d } )}
	\arrow["\alpha", from=1-1, to=1-3]
	\arrow[two heads, from=1-1, to=2-1]
	\arrow[two heads, from=1-3, to=2-3]
	\arrow["\beta", from=2-1, to=2-3]
	\arrow["{\rotatebox{270}{$\cong$}}", from=2-1, to=3-1]
	\arrow["{\rotatebox{270}{$\cong$}}", from=2-3, to=3-3]
	\arrow["\gamma"', from=3-1, to=3-3]
\end{tikzcd}
\]
We must show that $\alpha$ and $\gamma$ are surjective simultaneously.  This is equivalent to showing that $\alpha$ and $\beta$ are surjective simultaneously, which is true by Lemma \ref{lem:surjective_simultaneously}.
\end{proof}

\section{The spine of the lattice of $d$-torsion classes}
\label{sec:spine}

Recall from Remark \ref{rmk:cycles} the notion of a cycle in $\cM$.

\begin{Lemma}
\label{lem:deleting_an_indecomposable_from_a_set}
Assume that $\cM$ has no cycles.  

Let $\cX \subseteq \cM$ be a finite set of pairwise non-isomorphic indecomposable modules.  Then there exists $X \in \cX$ such that $X \in ( \cX \setminus X )^{ \perp }$.
\end{Lemma}

\begin{proof}
If the conclusion of the lemma were false, then each $X \in \cX$ would be outside $( \cX \setminus X )^{ \perp }$.  That is, for each $X \in \cX$ there would exist $X' \in \cX \setminus X$ and a homomorphism $X' \xrightarrow{} X$ that was not zero.  Since $X'$ and $X$ are different, hence non-isomorphic, the homomorphism would also not be an isomorphism.  Starting with an arbitrary $X_0 \in \cX\!$, we would hence be able to construct a sequence
\[
  \cdots \xrightarrow{} X_2 \xrightarrow{} X_1 \xrightarrow{} X_0
\]
with each $X_i$ in $\cX$ and each homomorphism not zero and not an isomorphism.  Since $\cX$ is finite, we would have $X_t = X_0$ for some $t \geqslant 1$, but this would contradict that $\cM$ has no cycles.
\end{proof}

\begin{Lemma}
\label{lem:deleting_an_indecomposable_from_a_d-torsion_class}
Assume that $\cM$ has no cycles.  

Let $\cV \subseteq \cW$ be a pair of splitting $d$-torsion classes differing by finitely many indecomposable modules.  

Pick a finite set $\cY \subseteq \cW \setminus \cV$ of pairwise non-isomorphic indecomposable modules such that $\add( \cY,\cV ) = \cW$.  Let $y \geqslant 0$ be the cardinality of $\cY\!$.

Then there exists a chain of splitting $d$-torsion classes in $\cM\!$,
\begin{equation}
\label{equ:lem:deleting_an_indecomposable_from_a_d-torsion_class:10}	
  \cV = \cW_y \subset \cdots \subset \cW_1 \subset \cW_0 = \cW\!,
\end{equation}
such that each is obtained by dropping one indecomposable module in $\cY$ from the class immediately above it.  That is, the inclusions are sharp, and there is a numbering $Y_1, \ldots, Y_y$ of the elements of $\cY$ such that $\add( Y_{ i+1 },\cW_{ i+1 } ) = \cW_i$ for $i \in \{\, 0, \ldots, y-1 \,\}$.
\end{Lemma}

\begin{proof}
Suppose that $\cW_t \subset \cdots \subset \cW_1 \subset \cW_0 = \cW$ and $Y_t, \ldots, Y_1$ have been constructed for some $t \geqslant 0$.  Since each $\cW_i$ is obtained by dropping the indecomposable module $Y_i$, we have
\begin{equation}
\label{equ:lem:deleting_an_indecomposable_from_a_d-torsion_class:20}
  \cW_t = \add( \cY \setminus \{\, Y_t, \ldots Y_1 \,\},\cV ).
\end{equation}
Note that if $t = 0$, then $\{\, Y_t, \ldots Y_1 \,\}$ is empty.

We claim that it is enough to assume $\cV \subset \cW_t$ and construct $\cW_{ t+1 }$ and $Y_{ t+1 }$.  Indeed, we can then iterate the construction, which at some point will give $\cV = \cW_t$ since $\cY$ is finite.  At this point, we will have constructed the chain \eqref{equ:lem:deleting_an_indecomposable_from_a_d-torsion_class:10}, and Equation \eqref{equ:lem:deleting_an_indecomposable_from_a_d-torsion_class:20} will imply $\cY \setminus \{\, Y_t, \ldots Y_1 \,\} = \emptyset$ whence $t = y$, so \eqref{equ:lem:deleting_an_indecomposable_from_a_d-torsion_class:10} has length $y$ as claimed.

So we assume $\cV \subset \cW_t$ and construct $\cW_{ t+1 }$ and $Y_{ t+1 }$.  Observe that $\cV \subset \cW_t$ combined with Equation \eqref{equ:lem:deleting_an_indecomposable_from_a_d-torsion_class:20} implies that $Y_t, \ldots, Y_1$ are not yet all the indecomposable modules in $\cY\!$.  Using Lemma \ref{lem:deleting_an_indecomposable_from_a_set}, pick $Y_{ t+1 } \in \cY \setminus \{\, Y_t, \ldots, Y_1 \,\}$ such that $Y_{ t+1 } \in ( \cY \setminus \{\, Y_{ t+1 }, \ldots, Y_1 \,\} )^{ \perp }$.  Define
\begin{equation}
\label{equ:lem:deleting_an_indecomposable_from_a_d-torsion_class:30}
  \cW_{ t+1 } = \add( \cY \setminus \{\, Y_{ t+1 }, \ldots, Y_1 \,\},\cV ).
\end{equation}

Equations \eqref{equ:lem:deleting_an_indecomposable_from_a_d-torsion_class:20} and \eqref{equ:lem:deleting_an_indecomposable_from_a_d-torsion_class:30} imply $\cW_{ t+1 } \subset \cW_t$ and $\add( Y_{ t+1 },\cW_{ t+1 } ) = \cW_t$, so to complete the proof, it is enough to show that $\cW_{ t+1 }$ is a splitting $d$-torsion class.  By Lemma \ref{lem:criterion_for_splitting}(iii), it suffices to let an indecomposable module $M \in \cM$ be given and prove $M \in \cW_{ t+1 }$ or $M \in \cW_{ t+1 }^{ \perp }$.  Suppose $M \not\in \cW_{ t+1 }$.  Then by Equations \eqref{equ:lem:deleting_an_indecomposable_from_a_d-torsion_class:20} and \eqref{equ:lem:deleting_an_indecomposable_from_a_d-torsion_class:30} there are two possibilities: $M \cong Y_{ t+1 }$ or $M \not\in \cW_t$.  

If $M \cong Y_{ t+1 }$ then we know $M \in ( \cY \setminus \{\, Y_{ t+1 }, \ldots, Y_1 \,\} )^{ \perp }$.  We also know $M \not\in \cV$ whence $M \in \cV^{ \perp }$ by Lemma \ref{lem:criterion_for_splitting}(iii) since $\cV$ is a splitting $d$-torsion class.  It now follows from Equation \eqref{equ:lem:deleting_an_indecomposable_from_a_d-torsion_class:30} that we have $M \in \cW_{ t+1 }^{ \perp }$.  If $M \not\in \cW_t$ then we have $M \in \cW_t^{ \perp }$ by Lemma \ref{lem:criterion_for_splitting}(iii) since $\cW_t$ is a splitting $d$-torsion class.  It now follows from $\cW_{ t+1 } \subset \cW_t$ that we have $M \in \cW_{ t+1 }^{ \perp }$.
\end{proof}

Recall that the spine of a lattice was originally defined for trim lattices by Thomas in \cite[p.\ 257]{Thomas_An-analogue-of-distributivity-for-ungraded-lattices}.  It consists of the elements which belong to a chain of maximal length.

\begin{Theorem}
\label{thm:spine}
Assume that $\cM$ has no cycles and only finitely many isomorphism classes of indecomposable modules.  

Then the elements of the spine of the lattice of $d$-torsion classes of $\cM$ are precisely the splitting $d$-torsion classes.
\end{Theorem}

\begin{proof}
We let $m$ denote the number of isomorphism classes of indecomposable modules in $\cM$.

On the one hand, let $\cU \subseteq \cM$ be a splitting $d$-torsion class.  Applying Lemma \ref{lem:deleting_an_indecomposable_from_a_d-torsion_class} to the pairs $0 \subseteq \cU$ and $\cU \subseteq \cM$ produces two chains of splitting $d$-torsion classes in $\cM\!$.  Merging them gives a chain of splitting $d$-torsion classes in $\cM\!$,
\[
  0 \subset \cdots \subset \cU \subset \cdots \subset \cM\!,
\]
such that each is obtained by dropping one indecomposable module from the class immediately above it.  This property forces the chain to be of length $m$, and to be of maximal length in $\dtors\,\cM\!$, the lattice of $d$-torsion classes in $\cM$.  Since $\cU$ is in the chain, it is an element of the spine of $\dtors\,\cM$.  

Observe that as a byproduct, we learned that a chain of maximal length in $\dtors\,\cM$ has length $m$.

On the other hand, let $\cU \subseteq \cM$ be a $d$-torsion class which is an element of the spine of $\dtors\,\cM\!$.  Then there exists a chain of $d$-torsion classes in $\cM\!$,
\[
  0 = \cW_m \subset \cdots \subset \cW_1 \subset \cW_0 = \cM\!,
\]
which has maximal length, such that $\cU$ is one of the $\cW_i$.  By the previous part of the proof, the length of the chain is indeed $m$, and this implies that each class in the chain is obtained by dropping one indecomposable module from the class immediately above it.  To prove that $\cU$ is a splitting $d$-torsion class, it is enough to prove that so is each $\cW_i$.  Assume not.  Then it is easy to see that there exists $t$ such that $\cW_t$ is splitting but $\cW_{ t+1 }$ is not.  By Theorem \ref{thm:generalisation_of_ASS_VI1.7}(ii) there is an indecomposable module $W \in \cW_{ t+1 }$ such that
\begin{equation}
\label{equ:thm:spine:10}
  \tau_d^-W \not\in \cW_{ t+1 }.
\end{equation}
However, we know $\cW_{ t+1 } \subset \cW_t$, so $W$ is in the splitting $d$-torsion class $\cW_t$ whence Theorem \ref{thm:generalisation_of_ASS_VI1.7}(ii) implies 
\begin{equation}
\label{equ:thm:spine:20}
  \tau_d^-W \in \cW_t.
\end{equation}
There is an indecomposable module $Y \in \cM$ such that 
\begin{equation}
\label{equ:thm:spine:25}
  \add( Y,\cW_{ t+1 } ) = \cW_t,
\end{equation}
and Equations \eqref{equ:thm:spine:10} and \eqref{equ:thm:spine:20} imply $Y \cong \tau_d^-W$.  Hence the $d$-Auslander--Reiten sequence starting at $W$ is
\begin{equation}
\label{equ:thm:spine:30}
  0
  \xrightarrow{} W
  \xrightarrow{} M^1
  \xrightarrow{} \cdots
  \xrightarrow{} M^d
  \xrightarrow{} Y
  \xrightarrow{} 0.
\end{equation}
The sequence is minimal, see \cite[def.\ 3.1]{Iyama_Higher-dimensional-AR-theory} and \cite[rmk.\ 4.5]{Fedele_d-Auslander-Reiten-sequences_in_subcategories}, so since we have $W \in \cW_{ t+1 } \subset \cW_t$ and $Y \cong \tau_d^-W \in \cW_t$, we get $M^1 \in \cW_t$ since $\cW_t$ is closed under $d$-extensions; see \cite[lem.\ 3.8(1) and prop.\ 3.11]{August-et-al_A-characterisation-of-higher-torsion-classes}.  Now note that $Y$ cannot be a direct summand of $M^1$.  Indeed, \cite[lem.\ 3.8]{Pasquali_Tensor-products-of-n-complete-algebras} implies that each differential of \eqref{equ:thm:spine:30} is non-vanishing on each direct summand of an object, so if $Y$ were a direct summand of $M^1$, then \eqref{equ:thm:spine:30} would provide a cycle from $Y$ to $Y$ in $\cM$.  We have learned that $M^1 \in \cW_t$ and that $Y$ is not a direct summand of $M^1$.  By Equation \eqref{equ:thm:spine:25} this implies $M^1 \in \cW_{ t+1 }$.  But $\cW_{ t+1 }$ is closed under $d$-quotients by \cite[lem.\ 3.8(2) and prop.\ 3.11]{August-et-al_A-characterisation-of-higher-torsion-classes}, so we get $M^2, \ldots, M^d, Y \in \cW_{ t+1 }$.  By $Y \cong \tau_d^-W$ this contradicts Equation \eqref{equ:thm:spine:10}.
\end{proof}

\begin{Example}
[Dynkin type $A$]
\label{exa:Type_A}
Let $n \geqslant 1$ be an integer.  The $d$-Auslander algebra $A_n^d$ of Dynkin type $A_n$ was introduced in \cite[sec.\ 6.2]{Iyama_Cluster-tilting-for-higher-Auslander-algebras} based on the general theory developed in \cite{Iyama_Auslander-correspondence}.  There is a unique $d$-cluster tilting subcategory $\cM_n^d \subseteq \mod\,A_n^d$, whose lattice of $d$-torsion classes was investigated in \cite[sec.\ 5.5]{August-et-al_A-characterisation-of-higher-torsion-classes}.

The conditions of Theorem \ref{thm:spine} are satisfied by the concrete description of $\cM_n^d$ provided in \cite[thm./constr.\ 3.4 and thm.\ 3.6(3)]{Oppermann-Thomas_Higher-dimensional-cluster-combinatorics-and-representation-theory}.

Hence Theorem \ref{thm:spine} implies that the elements of the spine of the lattice of $d$-torsion classes of $\cM_n^d$ are precisely the splitting $d$-torsion classes.
\end{Example}

Note that the same method can also be applied to the higher Nakayama algebras of type $A$ introduced in \cite[def.\ 2.17]{Jasso-Kuelshammer-Psaroudakis-Kvamme_Higher_Nakayama_algebras_I} by Jasso, K\"{u}lshammer, Kvamme, and Psaroudakis.

\medskip
\noindent
{\bf Statement on AI.}
All results in the paper were originally proved by the authors, and all parts of the paper were written by them.  The proof of Theorem \ref{thm:generalisation_of_ASS_VI1.7} was improved by ideas supplied by Microsoft Copilot.  Substantial help with the references was provided by ChatGPT Pro 5.6.

\medskip
\noindent
{\bf Acknowledgement.}
This work was supported by an International Mobility Grant from the Research Council of Norway (project number 357034) and a NOVA Grant from Aarhus University Research Foundation (grant AUFF-E-2024-9-43).

\end{document}